\documentclass[12pt]{article}
\input epsf.tex
\usepackage{amsthm, amssymb, amsmath}
\newtheorem{Theorem}{Theorem}
\newtheorem{lemma}{Lemma}

\newtheorem{remark}{Remark}

\begin{document}
\title{\bf{Logistic approximations and their consequences to bifurcations patterns and long-run dynamics}}
\author{Torsten Lindstr\"{o}m \\
Department of Mathematics \\
Linnaeus University \\
SE-35195 V\"{a}xj\"{o}, SWEDEN\\
\\
Yuanji Cheng\\
School of Technology\\
Malm\"{o} University\\
SE-20506 Malm\"{o}, SWEDEN}
\date{}
\maketitle

\begin{abstract}
On infinitesimally short time interval various processes contributing to population change tend to operate independently so that we can simply add their contributions (Metz and Diekmann (1986)\nocite{metzdiek}). This is one of the cornerstones for differential equations modeling in general. Complicated models for processes interacting in a complex manner may be built up, and not only in population dynamics. The principle holds as long as the various contributions are taken into account exactly. In this paper we discuss commonly used approximations that may lead to dependency terms affecting the long run qualitative behavior of the involved equations. We prove that these terms do not produce such effects in the simplest and most interesting biological case, but the general case is left open.
\end{abstract}

\section{Introduction}
A number of ecological phenomena can be studied under chemostat conditions, cf. Smith and Waltman (1995)\nocite{smithchemo}. For instance, a lake ecology may have rivers transferring a limiting resource to the lake and rivers diluting this resource. A phenomenological model containing such a situation is given by
\begin{eqnarray}
\dot{s}&=&CD-Ds-\frac{axs}{1+abs},\nonumber\\
\dot{x}&=&\frac{maxs}{1+abs}-Dx-\frac{Axy}{1+ABx},\label{chemostat}\\
\dot{y}&=&\frac{AMxy}{1+ABx}-Dy.\nonumber
\end{eqnarray}
Here $s>0$ is the substrate, $x>0$ is the prey having the substrate $s$ as its limiting resource and $y>0$ is a predator feeding on the prey $x$. The parameters $C>0$, $D>0$, $a>0$, $b>0$, $m>0$, $A>0$, $B>0$, and $M>0$ stand for concentration, dilution rate, search rate for the prey, handling time for the prey (cf. Holling (1959)\nocite{Holling.CanEnt:91}), conversion factor for the prey, search rate for the predator, handling time for the predator and conversion factor for the predator, respectively.

A logistic approximation is often made when studying (\ref{chemostat}) and related systems, cf. Kooi, Boer, and Kooijman (1998)\nocite{Kooi.BoMB:60} and references therein. More precisely, consider the function
\begin{displaymath}
H(s,x,y)=ms+x+\frac{y}{M}-mC.
\end{displaymath}
It satisfies $\dot{H}=-DH$ meaning that the surface $H=0$ is asymptotically invariant for (\ref{chemostat}). A study of the system on this surface allows for reducing it to a planar predator-prey system as follows
\begin{eqnarray}
\dot{x}&=&\frac{ax(mC-x-\frac{y}{M})}{1+\frac{ab}{m}(mC-x-\frac{y}{M})}-Dx-\frac{Axy}{1+ABx}\nonumber\\
\dot{y}&=&\frac{AMxy}{1+ABx}-Dy
\label{surface_pp}
\end{eqnarray}
If $b\neq 0$, then the growth function is non-logistic and given by
\begin{equation}
h(x)=\frac{ax(mC-x)}{1+\frac{ab}{m}(mC-x)}
\label{nonlogi}
\end{equation}
in the absence of predators. Its vertical asymptote is located at $x=mC+\frac{m}{ab}$, so it is non-negative in $[0,mC]$. It has exactly one extreme value on this interval and it is given by
\begin{displaymath}
x=\frac{mC(1+abC)}{1+abC+\sqrt{1+abC}}.
\end{displaymath}
This value tend to $mC/2$ as $b\rightarrow 0$. Thus, $h$ is unimodal on the relevant interval $[0,mC]$ and no qualitative differences can occur in the single species case when doing logistic approximations. Equivalents of the predator-prey system (\ref{surface_pp}) have been studied in Smith and Waltman (1995)\nocite{smithchemo} and Kuang (1989)\nocite{kuang}. The results were that local stability implies global stability and that uniqueness of limit cycles was proved for a certain range of parameters. However, it is still not known whether the limit cycle is unique for all parameters of the system.

If the handling time for the prey is small $b\approx 0$, the denominator in the first term of the first equation in (\ref{surface_pp}) vanishes and a logistic approximation is often made for the growth rate of the prey feeding on the substrate. We get
\begin{eqnarray}
\dot{x}&=&ax(mC-x)-Dx-\frac{axy}{M}-\frac{Axy}{1+ABx},\nonumber\\
\dot{y}&=&\frac{AMxy}{1+ABx}-Dy\label{ppsyst}.
\end{eqnarray}
In this context, the system
\begin{eqnarray}
\dot{x}&=&ax(mC-x)-Dx-\frac{Axy}{1+ABx},\nonumber\\
\dot{y}&=&\frac{AMxy}{1+ABx}-Dy\label{ppknown}.
\end{eqnarray}
has often been used as an approximation of the system (\ref{ppsyst}). Note that the coupling term $\frac{axy}{M}$ is set to zero in this procedure, too. This is entirely in concordance with well-established principles in differential equation modeling. As soon as logistic growth is accepted for the prey-population $x$, then one is expected to add the contributions of the predation process to the resulting equations. The equations (\ref{ppknown}) contain the result of such a procedure. The long-run qualitative dynamical properties including possibilities for limit cycle bifurcations of (\ref{ppknown}) are well-known trough a number of well-established results (cf. Cheng, Hsu, and Lin (1981)\nocite{chengfyra}, Liou and Cheng (1988)\nocite{chengtre} and Kuang and Freedman (1988)\nocite{kuangfyra}). The objective of this paper is to elucidate the influence of the above coupling term on those results. The numerical part of the work by Kooi, Boer, and Kooijman (1998)\nocite{Kooi.BoMB:60} reported differences in fixed point bifurcations for still longer food-chains as consequences of similar simplifications.

The methods used in this paper are the Lyapunov functions introduced by Ardito and Ricciardi (1995)\nocite{Ardito.JoMB:33} and LaSalle´s (1960)\nocite{LaSalle.IRETCT:7} theorem, cf. Hale (1969)\nocite{hale}, Hirsch, Smale, and Devaney (2013)\nocite{hirschsmaledevaney}, and Wiggins (2003)\nocite{wiggins2003}. We use the Kuang and Freedman (1988)\nocite{kuangfyra} version Zhang's (1986)\nocite{ZhangZF.ApplAnal:23} theorem (see also Cherkas and Zhilevich (1970)\nocite{Cherkas.DE:6}) for proving uniqueness of limit cycles. The result is that the qualitative properties of the bifurcation diagram are robust with respect to the coupling term mentioned above.

However, it is not straightforward to do the estimates required for application of these methods in particular cases. We have therefore limited the presentation here to the simplest and most interesting case (Logistic growth combined with Holling (1959)\nocite{Holling.CanEnt:91} specialist predator functional response). In the presentation of the proofs we have focused on finishing the relevant inequalities to the best forms whereas chains of equalities that involve extensive but elementary algebra have been given low priority. Our impression is that equalities are more straightforward to check than inequalities.

\section{Reparametrization}

We begin our study of the qualitative differences between the systems (\ref{ppknown}) and (\ref{ppsyst}) by reparametrizing the systems in order to eliminate some of the parameters involved (cf. Keener (1983)\nocite{keener}). We put
\begin{eqnarray*}
\xi&=&\frac{a}{amC-D}x\\
\eta&=&\frac{A}{amC-D}y\\
\tau&=&(amC-D)t\\
\epsilon&=&\frac{a}{MA}\\
\beta&=&\frac{AB}{a}(amC-D)\\
\mu&=&\frac{A}{a}(M-BD)\\
\lambda&=&\frac{Da}{A(amC-D)(M-BD)}
\end{eqnarray*}
and the system under consideration takes the form
\begin{eqnarray}
\dot{\xi}&=&\xi(1-\xi)-\epsilon\xi\eta-\frac{\xi\eta}{1+\beta\xi},\nonumber\\
\dot{\eta}&=&\eta\mu\frac{\xi-\lambda}{1+\beta\xi}.\label{pp_red}
\end{eqnarray}
The known system (\ref{ppknown}) corresponds to (\ref{pp_red}) with $\epsilon=0$ and seven parameters have been reduced to four. We assume in the rest of the paper that all parameters and variables are non-negative. The parameters $C$ and $D$ were those used as bifurcation parameters in Kooi, Boer, and Kooijman (1998)\nocite{Kooi.BoMB:60}.

\section{Isocline form and equilibria}

Most of the theorems and results regarding the system (\ref{pp_red}) are formulated for its isocline form
\begin{eqnarray}
\dot{\xi}&=&f(\xi)(F(\xi)-\eta),\nonumber\\
\dot{\eta}&=&\eta\psi(\xi),\label{pp_iso}
\end{eqnarray}
$\xi,\eta\geq 0$. The involved functions are required to meet some conditions. We specify them as follows:
\begin{itemize}
\item[(A-I)] $f$, $\psi$, and $F$ are $C^1([0,\infty))$,
\item[(A-II)] $f(0)=0$, $f(\xi)>0$ for $\xi>0$,
\item[(A-III)] $(\xi-1)F(\xi)<0$ for $\xi\neq 1$
\item[(A-IV)] $(\xi-\lambda)\psi(\xi)>0$, $\xi\neq\lambda>0$.
\end{itemize}
Condition (A-I) grants the local existence and uniqueness of solutions. In fact, only Lipschitz conditions are needed for this conclusion. However, we may need continuous derivatives of higher orders for some results below. These conditions are valid anyway for the specific system (\ref{pp_red}). Indeed, for (\ref{pp_red}), we have
\begin{eqnarray}
f(\xi)&=&\epsilon\xi+\frac{\xi}{1+\beta\xi},\nonumber\\
F(\xi)&=&\frac{(1+\beta\xi)(1-\xi)}{1+\epsilon+\epsilon\beta\xi},\label{spec_fcn}\\
\psi(\xi)&=&\mu\frac{\xi-\lambda}{1+\beta\xi}.\nonumber
\end{eqnarray}
These functions meet the conditions (A-I)-(A-IV) above. We first have
\begin{Theorem}
Assume \em (A-I)-(A-IV). \em The solutions of \em (\ref{pp_iso}) \em remain positive and bounded.
\label{dissipat_th}
\end{Theorem}
\begin{proof} By uniqueness of solutions, the solutions of (\ref{pp_iso}) cannot intersect the solutions at $\eta=0$ and $\xi=0$. Therefore, the solutions remain positive. For boundedness, note first that all solutions enter the region $0<\xi<1$ by (A-II)-(A-III) and that we have $\dot{\eta}<0$ for $0<\xi <\lambda$ by (A-IV). Next consider the function (cf. Lindstr\"{o}m (1993)\nocite{uppsats})
\begin{displaymath}
V(\xi,\eta)=\int_\lambda^\xi\frac{d\xi_\ast}{f(\xi_\ast)}+
\int_{F(\lambda)}^\eta\frac{d\eta_\ast}{\eta_\ast}.
\end{displaymath}
The last integral diverges as $\eta\rightarrow\infty$. Therefore, all level curves of this function intersect the line $\xi=\lambda$. The derivative along the solutions of (\ref{pp_iso}) for the above function is
\begin{displaymath}
\dot{V}=F(\xi)-\eta+\psi(\xi).
\end{displaymath}
Since $F(\xi)+\psi(\xi)$ is bounded on $\lambda<\xi<1$, this quantity is negative for $\eta$ large enough. That is, select $\kappa$ such that $\dot{V}<0$ for $V(\xi,\eta)>\kappa$ and we have a trapping region ensuring bounded solutions according to Figure \ref{kurva}(a).
\end{proof}
\begin{remark}\em
Because all positive solutions enter the region $0<\xi<1$ by the above argument, we usually check the conditions of the theorems alluded to below in that interval.
\em\end{remark}
We make a summary of the local bifurcation results.
\begin{Theorem}
Assume \em (A-I)-(A-IV). \em If $\lambda\geq 1$, the system \em (\ref{pp_iso}) \em has two equilibria, the origin and $(1,0)$. The origin is a saddle and $(1,0)$ is globally asymptotically stable. If $0<\lambda<1$, the system \em (\ref{pp_iso}) \em has three equilibria, the origin, $(1,0)$, and $(\lambda, F(\lambda))$. The first two are saddles and the last one is asymptotically stable when $F^\prime(\lambda)<0$ and unstable when $F^\prime(\lambda)>0$.
\label{localstab}
\end{Theorem}
\begin{proof}
Existence and number of equilibria follow from (A-II)-(A-IV). The Jacobian matrix of (\ref{pp_iso}) takes the form
\begin{displaymath}
J(\xi,\eta)=\left(\begin{array}{cc}
f^\prime(\xi)F(\xi)+f(\xi)F^\prime(\xi)-\eta f^\prime(\xi) & -f(\xi)\\
\eta\psi^\prime(\xi) & \psi(\xi)
\end{array}\right)
\end{displaymath}
meaning that the Jacobian matrixes evaluated at the first two equilibria take diagonal or triangular forms. Figure \ref{kurva}(b) demonstrates the techniques used for proving global asymptotic stability in the case $\lambda\geq 1$. The stability properties of the last equilibrium follows from the trace-determinant criterion, see e. g. Hirsch, Smale, and Devaney (2013)\nocite{hirschsmaledevaney} or Strang (1988)\nocite{Strang.Linear}.
\end{proof}
The stability properties of the last equilibrium are usually referred to as the Rosenzweig-MacArthur (1963)\nocite{rosenzweig} criterion. It shows that the properties of $F^\prime$ are important. Returning to (\ref{spec_fcn}), this function can be written in two forms
\begin{equation}
F^{\prime}(\xi)=\frac{\beta-1-\epsilon-2\beta\xi-2\beta\epsilon\xi-\beta^2\epsilon\xi^2}{(1+\epsilon+\epsilon\beta\xi)^2}=\frac{-(\xi-\xi_+)(\xi-\xi_-)}{\epsilon\left(\xi-\frac{\xi_++\xi_-}{2}\right)^2},
\label{fprime}
\end{equation}
where the latter one makes use of the zeroes
\begin{displaymath}
\xi_{\pm}=\frac{-1-\epsilon\pm\sqrt{1+\epsilon+\beta\epsilon}}{\beta\epsilon}
\end{displaymath}
of the numerator. We note that $\xi_{-}<0$. Moreover, $\xi_+$ is a strictly increasing function of $\beta$ that takes the value zero when $\beta=1+\epsilon$. Thus, the numerator of (\ref{fprime}) has a positive zero if $\beta>1+\epsilon$, otherwise not.

Theorem \ref{localstab} implies that $\lambda$ can be used as a bifurcation parameter and $\lambda\geq 1$ implies $(1,0)$ locally stable, $\max(0,\xi_+)<\lambda\leq 1$ implies $(\lambda,F(\lambda))$ to be locally stable, and
\begin{equation}
0<\lambda<\xi_+
\label{all_eq_unst}
\end{equation}
implies that all equilibria are unstable. Standard analysis gives now that
\begin{displaymath}
\xi_+<\lim_{\epsilon\rightarrow 0}\xi_+=\frac{\beta-1}{2\beta}<\frac{1}{2}.
\end{displaymath}
This relates our results to (\ref{pp_red}) with $\epsilon=0$. The coupling term in the system (\ref{pp_red}) therefore stabilizes the system locally. If the term containing $\epsilon>0$ is present, then $\lambda$ must be smaller in order to ensure possibilities for all non-negative equilibria to be unstable. The local bifurcation routes are qualitatively the same (as $\lambda$ decreases, possible destabilization occurs), only quantitative differences occur (the parameter $\lambda$ must be still smaller to ensure destabilization to take place as $\epsilon>0$). This phenomenon is usually called the "Paradox of enrichment" in the biological literature, cf. Rosenzweig (1971)\nocite{Rosenzweig.Science:171}. We conclude this section by proving:
\begin{lemma}
Assume \em (A-I)-(A-IV) \em and $0<\lambda<1$. All solutions of \em (\ref{pp_iso}) \em starting in the positive quadrant must have some part of the cross-section $\xi=\lambda$ and $0<\eta\leq F(\lambda)$ in its $\omega$-limit set.
\label{cross_sec_thm}
\end{lemma}
\begin{proof}
By Theorem \ref{dissipat_th} all solutions remain bounded and positive. The unstable manifold of (1,0) can therefore not intersect the stable manifold of itself or the origin. The Butler-McGehee theorem (Smith and Waltman (1995)\nocite{smithchemo}) implies that no $\omega$-limit set can contain (1,0) or the origin. By the Poincar\'{e}-Bendixson theorem the $\omega$-limit set is either the fixed point $(\lambda,F(\lambda))$ or a limit cycle that must surround $(\lambda,F(\lambda))$ by index theory, see e. g. Jordan and Smith (1990)\nocite{jordan} or Dumortier, Llibre, and Art\'{e}s (2006)\nocite{Dumortierbook}. In both cases, trajectories starting in the positive cone have some part of the ray $\xi=\lambda$, $0<\eta\leq F(\lambda)$ in their limit set.
\end{proof}

\section{Global stability}

We now continue with global stability. Our results follow from the results in Smith and Waltman (1995)\nocite{smithchemo} in a limiting case but our method is different (and simpler) anyway. For $\theta\geq 0$, we introduce the range of Lyapunov functions
\begin{equation}
W_\theta(\xi,\eta) =\eta^\theta \int_{\lambda}^\xi\frac{\psi(\xi_\ast)}{f(\xi_\ast)} d\xi_\ast
+\int_{F(\lambda)}^{\eta}\eta_\ast^{\theta}\frac{\eta_\ast-F(\lambda)}{\eta_\ast}d\eta_\ast
\label{ArditoRiccLyap}
\end{equation}
see Ardito and Ricciardi (1995)\nocite{Ardito.JoMB:33} and Lindstr\"{o}m (2000)\nocite{szeged1}. We introduce
\begin{displaymath}
\bar{F}_\theta(\xi)=F(\lambda)-\theta\int_\lambda^\xi\frac{\psi(\xi_\ast)}{f(\xi_\ast)}d\xi_\ast
\end{displaymath}
and deduce that the total derivative of $W_\theta$ along the solutions of (\ref{pp_iso}) is given by
\begin{displaymath}
\dot{W}_\theta=\eta^{\theta}\psi(\xi)(F(\xi)-\bar{F}_\theta(\xi)).
\end{displaymath}
Therefore, the sign condition
\begin{equation}
\left(F(\xi)-\bar{F}_\theta(\xi)\right)(\xi-\lambda)>0,\: \xi\neq\lambda
\label{tknvillkor}
\end{equation}
implies that LaSalle's (1960)\nocite{LaSalle.IRETCT:7} invariance principle can be applied. We also note that $W_\theta$ is a first integral of
\begin{eqnarray}
\dot{\xi}&=&f(\xi)\left(\bar{F}_\theta(\xi)-\eta\right)\nonumber\\
\dot{\eta}&=&\eta\psi(\xi)\label{rotsyst}
\end{eqnarray}
and that (\ref{rotsyst}) is a rotation of the vector field (\ref{pp_iso}) if (\ref{tknvillkor}), $\xi,\eta>0$, cf. Ye et. al. (1986)\nocite{ye}. The sign condition (\ref{tknvillkor}) means basically that the prey-isoclines of the systems (\ref{pp_iso}) and (\ref{rotsyst}) are compared to each other. For $\theta=0$, all solutions (\ref{rotsyst}) in the positive quadrant are closed orbits, and hence the level curves of (\ref{ArditoRiccLyap}), see Figure \ref{kurva}(c). For $\theta>0$, (\ref{rotsyst}) has non-closed solutions in the positive cone. However, all solutions intersecting the ray $\xi=\lambda$, $0<\eta\leq F(\lambda)$ are closed orbits. See Figure \ref{kurva}(d) and cf. Lemma \ref{cross_sec_thm}.

If $F(\xi)-\bar{F}_\theta(\xi)$ decreases, then the sign-condition (\ref{tknvillkor}) is clearly satisfied. We differentiate and obtain the requirement (cmp. e. g. (\ref{spec_fcn}) and (\ref{fprime}))
\begin{eqnarray*}
F^\prime(\xi)-\bar{F}_\theta^\prime(\xi)&=&F^\prime(\xi)+\theta\frac{\psi(\xi)}{f(\xi)}\\
&=&\frac{-(\xi-\xi_+)(\xi-\xi_-)}{\epsilon\left(\xi-\frac{\xi_++\xi_-}{2}\right)^2}+
\frac{\theta\mu(\xi-\lambda)}{\epsilon\beta\xi\left(\xi-\frac{\xi_++\xi_-}{2}\right)}<0,\\
\end{eqnarray*}
$0<\xi<1$. If $\beta>1+\epsilon$, we can choose $\theta\mu=2\xi_+\beta>0$ and the above inequality is equivalent to
\begin{eqnarray*}
-\xi^3+(3\xi_++\xi_-)\xi^2+(-2\xi_+\xi_--\xi_+^2-2\xi_+\lambda)\xi+\xi_+\lambda(\xi_++\xi_-)&=&\\
2\xi_+(\xi_+-\lambda)\left(\xi-\frac{\xi_++\xi_-}{2}\right)-(\xi-\xi_+)^2\left(\xi-(\xi_-+\xi_+)\right)&<&0,
\end{eqnarray*}
on $0<\xi<1$. It is straightforward to check the equality above and it can be derived us\-ing a Taylor expansion around $\xi=\xi_+$. The inequality holds since $\lambda\geq\xi_+$ and
\begin{displaymath}
\xi_++\xi_-<\frac{\xi_++\xi_-}{2}<0<\xi<1.
\end{displaymath}
We can thus, formulate the following theorem
\begin{Theorem}
Consider \em (\ref{pp_iso}) \em with \em (\ref{spec_fcn}). \em If either
\begin{itemize}
\item[\em(a)\em] $\xi_+\leq 0<\lambda$ or
\item[\em(b)\em] $0<\xi_+\leq\lambda$
\end{itemize}
then $(\min(\lambda,1), \max(0,F(\lambda))) $ is globally asymptotically stable in the cone $\xi>0$, $\eta>0$.
\end{Theorem}
\paragraph{Proof}
The case $\lambda\geq 1$ has been treated in Theorem \ref{localstab}. Assume therefore, $0<\lambda<1$. Put $\theta=\max(0,2\xi_+\beta/\mu)\geq 0$ in (\ref{ArditoRiccLyap}).
If $\beta\leq 1+\epsilon$, then $F$ is decreasing on $[0,1]$ and the choice $\theta=0$ is sufficient. All level curves of (\ref{ArditoRiccLyap}) are closed in the positive quadrant, see Figure \ref{kurva}(c), and global stability follows from LaSalle's (1960)\nocite{LaSalle.IRETCT:7} invariance principle. If $\beta>1+\epsilon$ then we use $\theta=2\xi_+\beta/\mu>0$ in (\ref{ArditoRiccLyap}). All level curves of (\ref{ArditoRiccLyap}) in the positive quadrant are not closed, see Figure \ref{kurva}(d), but all non-negative trajectories of (\ref{pp_iso}) with (\ref{spec_fcn}) will have some part of the cross-section $\xi=\lambda$ and $0<\eta\leq F(\lambda)$ in their limit set by Lemma \ref{cross_sec_thm}. If the limit set is $(\lambda,F(\lambda))$ we are done, and if not, the trajectory enters a closed level curve of (\ref{ArditoRiccLyap}). Global stability follows by LaSalle's invariance principle again. \qed

\section{Uniqueness limit cycles}

If $F^\prime(\lambda)>0$, then there exists at least one limit cycle for (\ref{pp_iso}) with (A-I)-(A-IV) due to the Poincar{\'{e}}-Bendixson Theorem and Theorem \ref{dissipat_th}. If $\epsilon=0$, then this limit cycle is unique by Theorem 4.3 in Kuang and Freedman (1988)\nocite{kuangfyra} for (\ref{pp_iso}) with (\ref{spec_fcn}). The theorem basically shows that Zhang's (1986)\nocite{ZhangZF.ApplAnal:23} theorem for Li{\'{e}}nard equations can be applied to the predator-prey system (\ref{pp_iso}) after a sequence of coordinate transformations. The conditions are stronger than (A-I)-(A-IV) in order to validate the coordinate transformations and uniqueness.

We shall use Kuang and Freedman's (1988)\nocite{kuangfyra} theorem to prove that the limit cycle is unique for $\forall\epsilon>0$. This makes it more difficult to verify its crucial condition for all relevant parameter values. Using (\ref{fprime}) it takes the form
\begin{equation}
\frac{d}{d\xi}\left(\frac{f(\xi)F^\prime(\xi)}{\psi(\xi)}\right)=
\frac{d}{d\xi}\left(\frac{\xi\left(\beta-1-\epsilon-2\beta\xi-2\beta\epsilon\xi-\beta^2\epsilon\xi^2\right)}{\mu(\xi-\lambda)(1+\epsilon+\beta\xi\epsilon)}\right)\leq 0
\label{Kuang_cond}
\end{equation}
for $\xi\neq\lambda$ in $0\leq\xi\leq1$. The criterion states that the divergence integrated around an assumed outer limit cycle is smaller than if it was integrated around an assumed inner limit cycle. The remaining conditions are satisfied for our selection of involved functions (\ref{spec_fcn}). It is sufficient to check that the numerator of (\ref{Kuang_cond}) is negative. In our case, this condition is equivalent to
\begin{eqnarray}
-\epsilon^2\beta^3\xi^4-2\epsilon\beta^2\left(1+\epsilon-\lambda\epsilon\beta\right)\xi^3&-&\nonumber\\
\beta\left((1+\epsilon)(2+\epsilon)+\epsilon\beta-5\lambda\epsilon\beta(1+\epsilon)\right)\xi^2&+&\label{Kuang_cond_eq}\\
4\lambda\beta\left(1+\epsilon\right)^2\xi-\lambda(1+\epsilon)\left(\beta-(1+\epsilon)\right)&\leq& 0\nonumber
\end{eqnarray}
that after Taylor expansion around $\xi=\lambda$ is equivalent to
\begin{eqnarray}
\beta\epsilon\lambda\left(\lambda+\frac{1+\epsilon}{\beta\epsilon}\right)(\lambda-\xi_-)(\lambda-\xi_+)&+&\nonumber\\
2\beta\epsilon\lambda(\lambda-\xi_-)(\lambda-\xi_+)(\xi-\lambda)&-&\label{org_Taylor_exp}\\
(1+\epsilon)\left(\lambda+\frac{1+\epsilon}{\beta\epsilon}+\frac{1}{\beta\epsilon}\right)(\xi-\lambda)^2&-&\nonumber\\
(\xi-\lambda)^2\left(1+2\beta\epsilon\left(\lambda+\frac{1+\epsilon}{\beta\epsilon}\right)(\xi-\lambda)+\beta\epsilon(\xi-\lambda)^2\right)&\leq&0.\nonumber
\end{eqnarray}
Straightforward computations verify the expressions above. The last term of (\ref{org_Taylor_exp}) is negative definite since
\begin{eqnarray*}
1+2\beta\epsilon\left(\lambda+\frac{1+\epsilon}{\beta\epsilon}\right)(\xi-\lambda)+\beta\epsilon(\xi-\lambda)^2&=&\\
1-2(1+\epsilon)\lambda-\beta\epsilon\lambda^2+2(1+\epsilon)\xi+\beta\epsilon\xi^2&>&\\
1-2(1+\epsilon)\lambda-\beta\epsilon\lambda^2&>&0.
\end{eqnarray*}
We remark that the last inequality holds because
\begin{displaymath}
\lambda<\xi_+<\frac{-1-\epsilon+\sqrt{1+2\epsilon+\epsilon^2+\beta\epsilon}}{\beta\epsilon}.
\end{displaymath}
It remains to prove that the three first terms of (\ref{org_Taylor_exp}) form a negative contribution. This quadratic expression has negative leading term and may be estimated from above by lines tangent to it at $\xi=0$ and $\xi=\lambda$. These two lines are given by
\begin{eqnarray*}
L_0(\xi)&=&\lambda\frac{1+\epsilon}{\epsilon\beta^2}\left(4\beta(1+\epsilon)\xi-\left(\beta-(1+\epsilon)\right)\right)\\
&=&\lambda\left(1+\epsilon\right)\left(-2(\xi_++\xi_-)\xi+\xi_+\xi_-\right)
\end{eqnarray*}
(from (\ref{Kuang_cond_eq}) and (\ref{fprime})) and
\begin{displaymath}
L_\lambda(\xi)=\beta\epsilon\lambda(\lambda-\xi_-)(\lambda-\xi_+)\left(-\lambda+\frac{1+\epsilon}{\beta\epsilon}+2\xi\right)
\end{displaymath}
(from (\ref{org_Taylor_exp})), respectively. We evaluate $L_0$ and $L_\lambda$ at points allowing for simple computations as follows
\begin{eqnarray*}
L_0\left(\frac{\xi_+}{2}\right)&=&-(1+\epsilon)\lambda\xi_+^2<0,\\
L_\lambda\left(\frac{\lambda}{2}\right)&=&\lambda(\lambda-\xi_-)(\lambda-\xi_+)\left(1+\epsilon\right)<0.
\end{eqnarray*}
Now, since $\lambda<\xi_+$, the three first terms in (\ref{org_Taylor_exp}) give a negative total contribution. Therefore, the limit cycle existing for $0<\lambda<\xi_+$ is unique. We have thus, proved:

\begin{Theorem}
Consider \em (\ref{pp_iso}) \em with \em (\ref{spec_fcn}). \em If $0<\lambda<\xi_+$ then \em (\ref{pp_iso}) \em possesses a unique limit cycle that is stable.
\end{Theorem}

\section{Discussion}

We have considered the system (\ref{pp_red}) in this paper. This system was derived from a three-trophic level (resource-prey-predator) chemostat model possessing an asymptotically invariant plane that allows reduction to two trophic levels (prey and predator). If the prey is unsaturated, logistic growth follows but a coupling term affecting the predators functional response cannot be removed. We prove that $\lambda$ can be used as a bifurcation parameter and that $(1,0)$ is globally stable for $\lambda\geq 1$, $(\lambda,F(\lambda))$ is globally stable for $\xi_+\leq\lambda\leq 1$, and that a unique stable limit cycle exists for $0<\lambda<\xi_+$. The bifurcations occurring at $\lambda=1$, and at $\lambda=\xi_+$ are a transcritical bifurcation, and a supercritical Hopf bifurcation, respectively, see e. g. Guckenheimer and Holmes (1983)\nocite{guck} or Kuznetsov (1998)\nocite{Kuznetsovbook}. The result was that destabilizing bifurcation pattern is identical to the well-known case without coupling term, but that smaller values of $\lambda$ are needed for destabilization.

The analysis shows it to be far from trivial in general to deduce what kind of effects such coupling terms might have on the existence of possible limit cycles. Obvious cases that are left open are saturated prey (giving rise to a non-logistic growth function that depends on the predator), generalist predator functional responses (May (1976)\nocite{may:second} or discontinuous variants Krebs and Davies(1993)\nocite{Krebs.Introduction}), and group defence cases cf. Geritz and Gyllenberg (2012,2013)\nocite{Geritz.JTB:314,Geritz.JoMB:66}. In several cases, results granting global stability or uniqueness of limit cycles exist if logistic growth is assumed and coupling terms are neglected, see e. g. Lindstr\"{o}m (1989)\nocite{acta}, Hwang (2003,2004)\nocite{JMAA.Hwang:281,JMAA.Hwang:290}, and Qiu and Xiao (2013)\nocite{Qiu.NA.RWA:14}. Some authors are, however, critical to the foundation for some of the investigated functional responses, see e. g. Oksanen, Moen, and Lundberg (1992)\nocite{oksanentio}, Diehl, Lundberg, Gardfjell, Oksanen, and Persson (1993)\nocite{oksanenelva}, and Gleeson (1994)\nocite{Gleeson.Ecology:75}.

We remark that examples demonstrating complicated limit cycles bifurcations have been created without considering the dependency terms considered here, see e. g. Gonz{\'{a}}lez-Olivarez and Rojas-Palma (2011)\nocite{Gonzalez-Olivares.BoMB:73}. It is therefore not evident that uniqueness of limit cycles is granted in predator-prey systems when they exist or that local stability implies global stability. Since such results are not granted even when the dependency terms are neglected, makes it even more complicated to elucidate the long-term dynamical effects of these terms in the general case.

Our analysis shows that a detailed investigation of many of these cases require further development of the methods, but this is usually difficult to do as long as a reasonable set of well-known special cases does not exist. Despite that many of the available methods provide possibilities for analyzing the different cases by elementary calculus methods, the arising expressions usually give rise to extensive algebraic manipulations when doing the required estimates. In this case we solved many of these problems through a careful selection of the quantities computed.

\begin{figure}
\epsfxsize=138mm
\begin{picture}(232,232)(0,0)
\put(0,0){\epsfbox{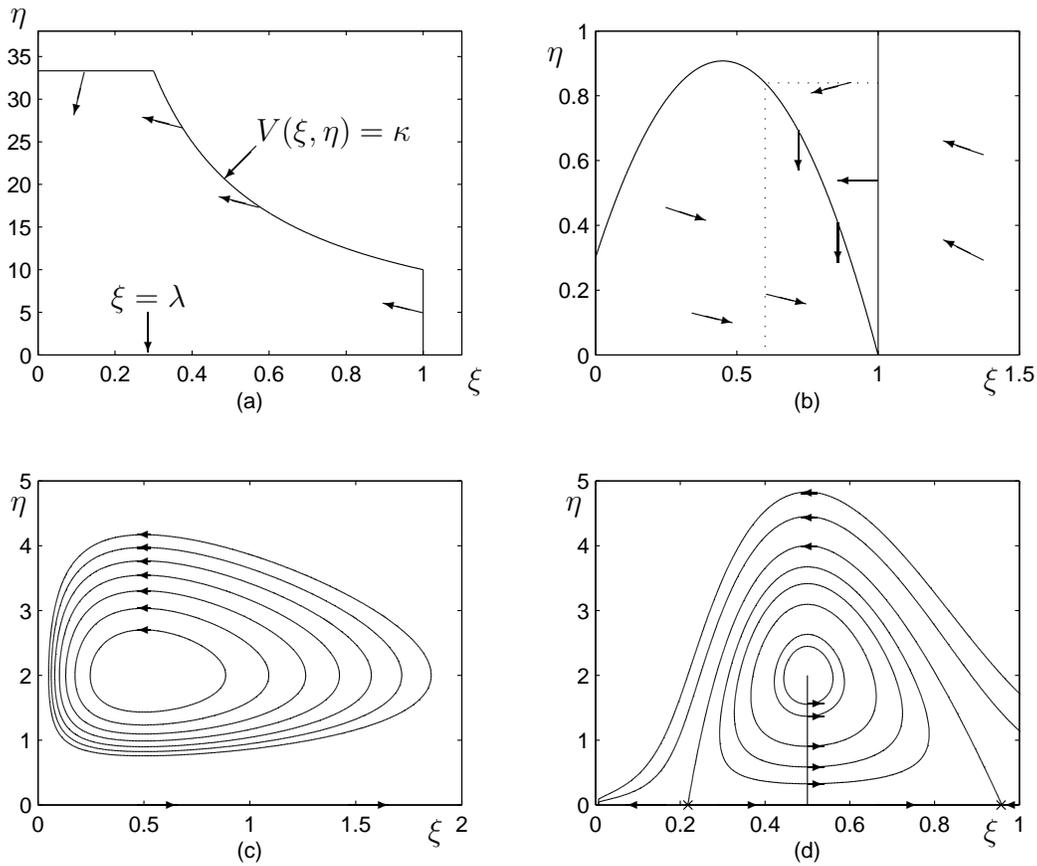}}
\put(95,275){$V(\xi,\eta)=\kappa$}
\put(95,273){\vector(-1,-1){12}}
\put(40,212){$\xi=\lambda$}
\put(54,210){\vector(0,-1){15}}
\put(175,181){$\xi$}
\put(2,320){$\eta$}
\put(158,210){\vector(-4,1){15}}
\put(96,250){\vector(-4,1){15}}
\put(67,280){\vector(-4,1){15}}
\put(30,301){\vector(-1,-4){4}}
\put(260,210){\vector(4,-1){15}}
\put(250,250){\vector(3,-1){15}}
\put(320,297){\vector(-4,-1){15}}
\put(330,260){\vector(-1,0){15}}
\put(288,217){\vector(4,-1){15}}
\put(300,279){\vector(0,-1){15}}
\put(315,244){\vector(0,-1){15}}
\put(370,230){\vector(-2,1){15}}
\put(370,270){\vector(-3,1){15}}
\put(370,181){$\xi$}
\put(205,305){$\eta$}
\put(55,126){\vector(-1,0){5}}
\put(55,121){\vector(-1,0){5}}
\put(55,115.8){\vector(-1,0){5}}
\put(55,110.5){\vector(-1,0){5}}
\put(55,104.5){\vector(-1,0){5}}
\put(55,98){\vector(-1,0){5}}
\put(55,89.7){\vector(-1,0){5}}
\put(50,23.5){\vector(1,0){15}}
\put(130,23.5){\vector(1,0){15}}
\put(160,10){$\xi$}
\put(2,135){$\eta$}
\put(307,141.7){\vector(-1,0){6}}
\put(307,132.5){\vector(-1,0){6}}
\put(307,121.5){\vector(-1,0){6}}
\put(303.5,62){\vector(1,0){6}}
\put(303.5,57){\vector(1,0){6}}
\put(303.5,45.8){\vector(1,0){6}}
\put(303.5,38){\vector(1,0){6}}
\put(303.5,31.5){\vector(1,0){6}}
\put(250,23.5){\vector(-1,0){15}}
\put(270,23.5){\vector(1,0){15}}
\put(330,23.5){\vector(1,0){15}}
\put(383,23.5){\vector(-1,0){5}}
\put(370,10){$\xi$}
\put(212,135){$\eta$}
\end{picture}
\caption{(a) Trapping region formed for system (\protect\ref{pp_iso}) by the level curve $V(\xi,\eta)=\kappa$ in the interval $\lambda<\xi<1$. (b) The fixed point $(1,0)$ is globally asymptotically stable if $\lambda\geq 1$. (c) For $\theta=0$, all solutions of (\protect\ref{rotsyst}) are closed orbits. (d) For $\theta>0$ there may be non-closed solutions of (\protect\ref{rotsyst}), but all solutions passing the cross-section $\xi=\lambda$, $0<\eta\leq F(\lambda)$ are closed. The $\times$-marks denote two saddles in a heteroclinic loop.} \label{kurva}
\end{figure}

\bibliographystyle{abbrv}
\bibliography{artiklar}

\begin{thebibliography}{10}

\bibitem{Ardito.JoMB:33}
A.~Ardito and P.~Ricciardi.
\newblock Lyapunov functions for a generalized {G}ause-type model.
\newblock {\em Journal of Mathematical Biology}, 33:816--828, 1995.

\bibitem{chengfyra}
K.-S. Cheng, S.-B. Hsu, and S.-S. Lin.
\newblock Some results on global stability of a predator-prey system.
\newblock {\em Journal of Mathematical Biology}, 12:115--126, 1981.

\bibitem{Cherkas.DE:6}
L.~A. Cherkas and L.~I. Zhilevich.
\newblock Some criteria for the absence of limit cycles and for the existence
  of a single cycle.
\newblock {\em Differential Equations}, 6:891--897, 1970.

\bibitem{oksanenelva}
S.~Diehl, P.~A. Lundberg, H.~Gardfjell, L.~Oksanen, and L.~Persson.
\newblock \protect\em {D}aphnia\protect\em-phytoplankton interactions in lakes:
  is there a need for ratio-dependent consumer-resource models.
\newblock {\em The American Naturalist}, 142(6):1052--1061, 1993.

\bibitem{Dumortierbook}
F.~Dumortier, J.~Llibre, and J.~C. Artes.
\newblock {\em Qualitative Theory of Planar Differential Systems}.
\newblock Springer, 2006.

\bibitem{Geritz.JTB:314}
S.~A.~H. Geritz and M.~Gyllenberg.
\newblock A mechanistic derivation of the {D}e{A}ngelis-{B}eddington functional
  response.
\newblock {\em Journal of Theoretical Biology}, 314:106--108, 2012.

\bibitem{Geritz.JoMB:66}
S.~A.~H. Geritz and M.~Gyllenberg.
\newblock Group defence and the predator's functional response.
\newblock {\em Journal of Mathematical Biology}, 66:705--717, 2013.

\bibitem{Gleeson.Ecology:75}
S.~K. Gleeson.
\newblock Density dependence is better than ratio dependence.
\newblock {\em Ecology}, 75:1834--1835, 1994.

\bibitem{Gonzalez-Olivares.BoMB:73}
E.~Gonz{{\'{a}}}lez-Olivares and A.~Rojas-Palma.
\newblock Multiple limit cycles in a {G}ause type predator-prey model with
  {H}olling type {III} functional response and allee effect on prey.
\newblock {\em Bulletin of Mathematical Biology}, 73:1378--1397, 2011.

\bibitem{guck}
J.~Guckenheimer and P.~Holmes.
\newblock {\em Nonlinear Oscillations, Dynamical Systems, and Bifurcations of
  Vector Fields}.
\newblock Springer-{V}erlag, 1983.

\bibitem{hale}
J.~K. Hale.
\newblock {\em Ordinary differential equations}.
\newblock Pure and Applied Mathematics. John Wiley {\&} Sons, Inc, 1969.

\bibitem{hirschsmaledevaney}
M.~W. Hirsch, S.~Smale, and R.~L. Devaney.
\newblock {\em Differential Equations, Dynamical Systems, and an Introduction
  to Chaos}.
\newblock Academic Press, Oxford, 2013.

\bibitem{Holling.CanEnt:91}
C.~S. Holling.
\newblock Some characteristics of simple types of predation and parasitism.
\newblock {\em The Canadian Entomologist}, 91(7):385--398, 1959.

\bibitem{JMAA.Hwang:281}
T.-W. Hwang.
\newblock Global analysis of the predator-prey system with
  {B}eddington-{D}e{A}ngelis functional response.
\newblock {\em Journal of Mathematical Analysis and Applications},
  281:395--401, 2003.

\bibitem{JMAA.Hwang:290}
T.-W. Hwang.
\newblock Uniqueness of limit cycles of the predator-prey system with
  {B}eddington-{D}e{A}ngelis functional response.
\newblock {\em Journal of Mathematical Analysis and Applications},
  290:113--122, 2004.

\bibitem{jordan}
D.~W. Jordan and P.~Smith.
\newblock {\em Nonlinear Ordinary Differential Equations}.
\newblock Clarendon Press, Oxford, second edition, 1990.

\bibitem{keener}
J.~P. Keener.
\newblock Oscillatory coexistence in the chemostat.
\newblock {\em {SIAM} Journal of Applied Mathematics}, 43(5):1005--1018, 1983.

\bibitem{Kooi.BoMB:60}
B.~W. Kooi, M.~P. Boer, and S.~A. L.~M. Kooijman.
\newblock On the use of the logistic equation in models of food chains.
\newblock {\em Bulletin of Mathematical Biology}, 60:231--246, 1998.

\bibitem{Krebs.Introduction}
J.~R. Krebs and N.~B. Davies.
\newblock {\em An Introduction to Behavioral Ecology}.
\newblock Blackwell Scientific Publications, Oxford, third edition, 1993.

\bibitem{kuang}
Y.~Kuang.
\newblock Limit cycles in a chemostat related model.
\newblock {\em {SIAM} Journal of Applied Mathematics}, 49(6):1759--1767,
  December 1989.

\bibitem{kuangfyra}
Y.~Kuang and H.~I. Freedman.
\newblock Uniqueness of limit cycles in {G}ause-type models of predator-prey
  systems.
\newblock {\em Mathematical Biosciences}, 88:67--84, 1988.

\bibitem{Kuznetsovbook}
Y.~A. Kuznetsov.
\newblock {\em Elements of Applied Bifurcation Theory}.
\newblock Springer, New York, 1998.

\bibitem{LaSalle.IRETCT:7}
J.~P. La{S}alle.
\newblock Some extensions of {L}yapunovs second method.
\newblock {\em IRE Transactions of Circuit Theory}, {CT}-7:520--527, 1960.

\bibitem{acta}
T.~Lindstr{\"{o}}m.
\newblock A generalized uniqueness theorem for limit cycles in a predator-prey
  system.
\newblock {\em Acta Academiae Aboensis, Ser B}, 49(2):1--9, 1989.

\bibitem{uppsats}
T.~Lindstr{\"{o}}m.
\newblock Qualitative analysis of a predator-prey system with limit cycles.
\newblock {\em Journal of Mathematical Biology}, 31:541--561, 1993.

\bibitem{szeged1}
T.~Lindstr{\"{o}}m.
\newblock Global stability of a model for competing predators: the {A}rdito
  {\&} {R}icciardi {L}yapunov functional.
\newblock {\em Nonlinear Analysis}, 39:793--805, 2000.

\bibitem{chengtre}
L.-P. Liou and K.-S. Cheng.
\newblock Global stability of a predator-prey system.
\newblock {\em Journal of Mathematical Biology}, 26:65--71, 1988.

\bibitem{may:second}
R.~M. May, editor.
\newblock {\em Theoretical Ecology: Principles and Applications}.
\newblock Blackwell Scientific Publications, 1976.
\newblock (Second edition 1981).

\bibitem{metzdiek}
J.~A.~J. Metz and O.~Diekmann.
\newblock {\em The Dynamics of Physiologically Structured Populations}.
\newblock Springer-Verlag, 1986.

\bibitem{oksanentio}
L.~Oksanen, J.~Moen, and P.~A. Lundberg.
\newblock The time scale problem in exploiter-victim models: does the solution
  lie in ratio-dependent exploitation?
\newblock {\em The American Naturalist}, 140(6):938--960, 1992.

\bibitem{Qiu.NA.RWA:14}
Qiu{~}Xiao-xiao and Xiao{~}Hai-bin.
\newblock Qualitative analysis of {H}olling type {II} predator-prey systems
  with prey refuges and predator restricts.
\newblock {\em Nonlinear Analysis: Real World Applications}, 14(4), 2013.

\bibitem{Rosenzweig.Science:171}
M.~L. Rosenzweig.
\newblock Paradox of enrichment: Destabilization of exploitation ecosystems in
  ecological time.
\newblock {\em Science}, 171:385--387, 1971.

\bibitem{rosenzweig}
M.~L. Rosenzweig and R.~H. MacArthur.
\newblock Graphical representation and stability conditions of predator-prey
  interactions.
\newblock {\em The American Naturalist}, 97:209--223, 1963.

\bibitem{smithchemo}
H.~L. Smith and P.~Waltman.
\newblock {\em The Theory of the Chemostat: Dynamics of Microbial Competition}.
\newblock Cambridge University Press, 1995.

\bibitem{Strang.Linear}
G.~Strang.
\newblock {\em Linear Algebra and its Applications}.
\newblock Harcourt Brace Jovanovich, third edition, 1988.

\bibitem{wiggins2003}
S.~Wiggins.
\newblock {\em Introduction to Applied Nonlinear Dynamical Systems and Chaos}.
\newblock Springer, New York, 2003.

\bibitem{ye}
Ye{~}Y.-Q. et~al.
\newblock {\em Theory of Limit Cycles}.
\newblock American Mathematical Society, second edition, 1986.

\bibitem{ZhangZF.ApplAnal:23}
Zhang{~}Zhi-fen.
\newblock {P}roof of the {U}niqueness {T}heorem of {L}imit {C}ycles of
  {G}eneralized {L}i{\'{e}}nard {E}quations.
\newblock {\em Applicable Analysis}, 29:63--76, 1986.

\end{thebibliography}
\end{document}